\documentclass{amsart}

\usepackage{color,todonotes}
\newcommand{\tl}[1]{#1}

\usepackage{mathtools,amsfonts,esint,amsthm,amssymb,amsmath,enumerate}
\newcommand{\R}{\mathbb{R}}
\newcommand{\E}{\mathcal{E}}
\newcommand{\Per}{P}

\newcommand{\Div}{\operatorname{div}}

\usepackage[colorlinks=true,linkcolor=black,citecolor=black,filecolor=black,urlcolor=black,menucolor=black]{hyperref}
 
\newtheorem{theorem}{Theorem}
\newtheorem{lemma}{Lemma}
\newtheorem{corollary}{Corollary}
\theoremstyle{remark}
\newtheorem{remark}{Remark}

\author{Antonin Chambolle}
\address{A.C.: CMAP, Ecole Polytechnique, CNRS,
	91128 Palaiseau, France}
\email{antonin.chambolle@cmap.polytechnique.fr}
\author{Tim Laux}
\address{T.L.: University of Bonn, Hausdorff Center for Mathematics, Villa Maria, Endenicher Allee 62, D-53115 Bonn, Germany}
\email{tim.laux@hcm.uni-bonn.de}

\begin{document}

\title[Wasserstein flow of the perimeter]{Mullins-Sekerka as the Wasserstein flow\\ of the perimeter}

\begin{abstract}
  We prove the convergence of an implicit time discretization for the \tl{one-phase} Mullins-Sekerka equation\tl{, possibly with additional non-local repulsion,} proposed in
  [F.~Otto, Arch.\ Rational Mech.\ Anal.\ 141 (1998) 63--103].
  Our simple argument shows that the limit satisfies the equation in a distributional sense as well as an optimal energy-dissipation relation.
  The proof combines arguments from optimal transport, gradient flows \& minimizing movements, and basic geometric measure theory. 
  
   \medskip
  
  
  \medskip

  \noindent \textbf{Math.\ Subject Classification}: 35A15, 35R37, 49Q20, 76D27, 90B06, 35R35
\end{abstract}

\maketitle

\section{Introduction}
	The Mullins-Sekerka equation, see \eqref{eq:MS_1a_strong}--\eqref{eq:MS_2b_strong} below for its exact formulation \tl{with additional non-local repulsion},  is a well-studied mathematical model which, among other phenomena, describes a Hele-Shaw cell: A viscous ferro-fluid is confined to a thin region between two parallel horizontal plates. Applying a strong magnetic field in the vertical direction leads to two opposing forces: (i) due to surface tension, the fluid wants to decrease its surface area; (ii) the probe becomes magnetized by the field and the particles repel each other due to the induced magnetic field. 
	These two competing effects lead to the formation of intriguing patterns. 
	
	\medskip
	
        In this paper we construct weak solutions using an implicit time discretization proposed by F.~Otto in \cite{otto1998dynamics} \tl{for the Mullins-Sekerka equation, possibly with non-local repulsion}.
        Because of the gradient-flow structure of the equation, it is natural to consider minimizing movements, an implicit time discretization which comes as a sequence of variational problems \cite{de1992movimenti}.  The effective energy consists of two terms, (i) an attractive term due to surface tension, the total surface area of the lateral boundary of the region occupied by the fluid, and (ii) a non-local term due to the magnetic repulsion of the particles; see \eqref{eq:def energy} below.
	In \cite{otto1998dynamics} it has been observed that the dissipation functional  may be modeled by the Wasserstein distance, which arises in optimal mass transport; see \eqref{eq:def W2}. The Wasserstein distance  plays a crucial role for many diffusion equations as was pointed out by Jordan, Kinderlehrer, and Otto in the seminal work \cite{JKO98}, see also \cite{JKOPhys}.

	\medskip
	
	The main theorem of the present work is a refined version of the announced result \cite[Theorem 1]{otto1998dynamics}, for which a
        detailed proof was not provided.
	Our simple proof establishes the convergence of the approximations obtained from the minimizing movements scheme to a weak solution. We derive the Mullins-Sekerka equation  \eqref{eq:MS_1a_strong}--\eqref{eq:MS_2b_strong} in a distributional form, and using De Giorgi's variational interpolations \cite{AGS}, we show that the limit satisfies an optimal energy-dissipation relation.
	The convergence of the energies as $h\to0$, a well-known assumption known from the more difficult  case of mean curvature flow \cite{LucStu95}, is not necessary in our case. 
	In fact, our proof is much simpler and no regularity theory of almost minimal surfaces is needed. 
	\tl{Dropping the assumption of energy convergence, however, comes at the price of a weaker solution concept involving varifolds instead of sets of finite perimeter. 
	We need of course to assume in addition the convergence of the time-integrated energies in order  to recover the expected formulation for sets of finite perimeter.}
	It may be expected that this solution concept, \tl{at least in the case of energy convergence,} satisfies a weak-strong uniqueness principle similar to the ones in the forthcoming works by Fischer, Hensel, Simon, and one of the authors for multiphase mean curvature flow \cite{FHLS_MCF} and for the simpler \emph{two-phase} Mullins-Sekerka equation \cite{FHLS_MS}.
	
	\medskip
	
	There has been continuous interest in the Mullins-Sekerka equation and similar gradient flows, so we only briefly point out some of the most relevant results related to the present work. 
	Weak solutions to the two-phase Stefan problem have been constructed by Luckhaus \cite{luckhaus1999solutions}. In particular, Luckhaus discovered a hidden variational principle satisfied by his approximations, which allows to verify the convergence of the energies as $h\to0$.
	Luckhaus and Sturzenhecker \cite{LucStu95} 	constructed weak solutions of mean curvature flow and the two-phase Mullins-Sekerka equation conditioned on the convergence of the energies. 
	R\"{o}ger \cite{roger2005existence} was able to remove the assumption in the case of this two-phase Mullins-Sekerka equation by showing that the assumption may only be violated along flat parts of $\partial E$. 
	In the case of mean curvature flow, the assumption can be verified in very particular cases, like convex sets \cite{caselles2006anisotropic}, graphs \cite{logaritsch2016obstacle}, and mean convex sets \cite{guidotim}.
	For generalizations to the anisotropic case, which for mean curvarture flow has already been introduced by Almgren, Taylor, and Wang \cite{ATW93}, we refer the interested reader to Garcke and Schaubeck \cite{garcke2011existence} and Kraus \cite{kraus2011degenerate}.
	A variant relevant for image denoising has been introduced by  Carlier and Poon \cite{carlier2019total} who relax the constraint $\chi_E \in \{0,1\}$, which leads to the total variation flow. However, it seems that the convergence can only be proven under an additional assumption on the density.
	Glasner \cite{glasner2002diffuse} introduced a phase-field approximation to the one-phase Mullins-Sekerka equation and studied its convergence by formal asymptotic expansions. 
	\tl{While the analysis of the non-degenerate Cahn-Hilliard equation to the two-phase Mullins-Sekerka equation is by now well-understood, see for example the work of Chen \cite{chen1996global} and Alikakos, Bates, and Chen \cite{alikakos1994convergence}, there seems to be no result for this degenerate version.}
	Recently, also the computationally efficient thresholding scheme by Merriman, Bence, and Osher \cite{MBO92, MBO94} has been reinterpreted as a minimizing movements scheme by Esedo\u{g}lu and Otto \cite{EseOtt14}, which allowed one of the author together with Otto to prove conditional convergence results to multiphase mean curvature flow \cite{laux2016convergence,laux2017brakke}.
	Most recently, Jacobs, Kim, and M\'{e}sz\'{a}ros \cite{jacobs2019weak} introduced an interesting thresholding-type approximation for the Muskat problem and proved a similar (conditional) convergence result for their scheme.

	\medskip
	
	The paper is organized as follows: In \S \ref{sec:result} we recall the minimizing movements scheme and state our main result\tl{,} Theorem \ref{thm}, which will be proved in the following sections: \S \ref{sec:compactness} establishes the compactness; in \S \ref{sec:convergence} we recover the distributional equation for the limit and the optimal energy-dissipation relation; and \S \ref{sec:interpolation inequality} contains a simple nonlinear interpolation inequality and its proof.
	
	\section*{Acknowledgement}
	The authors would like to warmly thank Prof.\ Felix Otto, Max Planck Institute (MPI), Leipzig,
        for mentioning this problem, and for numerous discussions and suggestions.
        This work was started during Mathematisches Forschungsinstitut Oberwolfach's (MFO) workshop \#1746 and completed during MFO's workshop \#1904, as well as visits to MPI Leipzig. The hospitality and support of these institutions is also acknowledged. This project was partially funded by the Deutsche Forschungsgemeinschaft (DFG, German Research Foundation) under Germany's Excellence Strategy - EXC-2047/1 - 390685813.

\section{Statement of the main result}\label{sec:result}

The physical model \tl{under consideration} is described by the following \tl{system of} equations \tl{for an evolving set $(E(t))_{t\in(0,\infty)}$ and a velocity field $\mathbf{u} \colon \R^d \times (0,\infty) \to \R^d$}: The interface $\partial E(t)$ is transported by the fluid
\begin{equation}\label{eq:MS_1a_strong}
V= \mathbf{u}\cdot \nu \quad \text{on } \partial E(t)
\end{equation}
(throughout, $\nu$ denotes the \tl{\emph{outer}} unit normal to $\partial E(t)$); 
the fluid is incompressible
\begin{equation}\label{eq:MS_1b_strong}
\Div \mathbf{u} = 0  \quad \text{in }  E(t);
\end{equation}
the flow is irrotational, i.e., 
\begin{equation}\label{eq:MS_2a_strong}
\text{there exists } p \text{ such that } \mathbf{u} = -\nabla p \text{ in }E(t);
\end{equation}
and \tl{on the interface, the following balance-of-forces condition holds}
\begin{equation}\label{eq:MS_2b_strong}
p= H + 2 k\ast \chi_{E(t)} \quad \text{on } \partial E(t),
\end{equation}
\tl{where $H$ denotes the mean curvature of $\partial E(t)$ with the sign convention $H>0$ for convex $E(t)$ and $k$ is a non-negative, symmetric, and normalized convolution kernel $k\ge 0$, $k(-z)=k(z)$, and $\int k=1$.}

\tl{Since the velocity field $\mathbf{u}$ is divergence-free, any} smooth solution $E(t)$ is volume-preserving $\frac{d}{dt}|E|= 0$.
 More importantly, \tl{combining \eqref{eq:MS_1a_strong}--\eqref{eq:MS_2b_strong}, we see that any such evolution is} energy dissipating 
\begin{equation}\label{eq:MS_halfDGstrong}
\frac{d}{dt} \tl{ \left( \mathcal{H}^{n-1}(\partial E(t)) + \int_{E(t)} k * \chi_{E(t)} \,dx\right) }= -\int_{E(t)} |\mathbf{u}|^2\,dx \leq 0.
\end{equation}
More precisely, the above set of equations have a gradient-flow structure. 

\tl{Here, the} metric tensor $\int_E | \mathbf{u}|^2\,dx$ defined on divergence-free vector fields $\mathbf{u} \colon E \to \R^d$ is less degenerate than the one of the mean curvature flow $\int_{\partial E} V^2 dS$ defined on normal velocities $V$, but more degenerate than the one of the two-phase Mullins-Sekerka problem $\int_{\R^d} |\mathbf{u}|^2dx$, in which the ferro-liquid is assumed to be surrounded by another liquid of the same viscosity.

For more physical motivation, we refer to the introduction of  \cite{otto1998dynamics} and the references therein. 

\medskip

\tl{Let us recall} the implicit time discretization introduced by F.~Otto in \cite{otto1998dynamics}: Given a time-step size $h>0$, and initial conditions $E_0\subset \R^d$, for $n\ge1$\tl{,} find $E_n$ solving
\begin{equation}\label{eq:mm}
\min_E  \Big\{\frac{1}{2h}W_2^2(\chi_E,\chi_{E_{n-1}}) + \Per(E) + \int_E k*\chi_E \,dx \Big\}.
\end{equation}
Here $P(E):= \sup\{ -\int_E \Div\xi \,dx \colon \sup |\xi|\leq 1\}$ denotes the perimeter of $E\subset \R^d$ and
\begin{equation}\label{eq:def W2}
W_2^2(\chi_{E}, \chi_F) 
=  \inf \int_E |x-T(x)|^2 dx
=  \min \iint |x-y|^2 d\gamma(x,y)
\end{equation}
denotes the squared Wasserstein distance, where the infimum runs over all transport maps, i.e., volume preserving diffeomorphisms $T\colon \R^d\to \R^d $ such that $T_\sharp\chi_{E}\tl{(x)dx}=\chi_F\tl{(y)dy}$, and the minimum runs over all transport plans, i.e., finite measures $\gamma$ in $\R^d\times \R^d$ with marginals $\chi_E(x) dx$ and $\chi_F(y)dy$. 
\tl{Note that the Wasserstein distance in the minimization term automatically enforces the volume constrained $|E_n| = |E_{n-1}|$, just like the system of partial differential equations \eqref{eq:MS_1a_strong}--\eqref{eq:MS_2b_strong} above.} 

We denote \tl{the total energy by}
\begin{equation}\label{eq:def energy}
\E(E):= \Per(E)+\int_E k*\chi_E\, dx,
\end{equation}
and let $E_h(t) := E_{[t/h]}$ for $t\ge 0$. 
Our standing assumption on the initial conditions is
\begin{equation}\label{eq:standingassumption}
	  P(E_0) < \infty\quad \text{and}\quad 
	  \int_{E_0} (1+|x|^2) \,dx <\infty.
\end{equation}
 In particular $|E_0| <\infty$ and w.l.o.g.\ by scaling we may assume that $|E_0| = 1$.

	The main result is the following construction of solutions.
\begin{theorem}\label{thm}
  Let $E_0 \subset \R^d$ be initial conditions satisfying \eqref{eq:standingassumption} and let $E_h(t)$ be constructed as above. Then there exists a subsequence $h\downarrow0$\tl{,} an \tl{$L^1$-}continuous one-parameter
   family of finite perimeter sets $E(t)$ \tl{satisfying \eqref{eq:standingassumption},} and a vector field $\mathbf{u} \in L^2(\R^d\times (0,+\infty)\tl{;\R^d})$  such that 
% \[
%    \lim_{h\downarrow0} \int_0^T|E_h(t)\triangle E(t)| \, dt =0\quad \text{for all } T<+\infty
% \]
		\[
			\lim_{h\downarrow0} \sup_{t\in [0,T]} |E_h(t)\triangle E(t)|  =0\quad \text{for all } T<+\infty
		\]
		and 		
		\begin{equation}\label{eq:MS_1}
			-\int_0^{+\infty} \int_{E(t)} \left(\partial_t \zeta + \mathbf{u}\cdot \nabla \zeta \right) \,dx\,dt = \int_{E_0} \zeta(0)\,dx
		\end{equation}
		for all $\zeta\in C_0^\infty(\R^d\times[0,+\infty))$\tl{,} and $E(t)$ satisfies the optimal energy dissipation rate
		\begin{equation}\label{eq:MS-halfDG}
		 \E(E(T)) +\int_0^T \int_{E(t)} |\mathbf{u}(x,t)|^2 \,dx\,dt \leq \E(E_0) \quad \text{for \tl{all} } T>0.
		\end{equation}
		
		The measures $\mu^h := \delta_{\nu_{E_h(t)}} \otimes \left| \nabla \chi_{E_h(t)}\right|  dt$ converge, $\mu^h \rightharpoonup \mu = \mu_t \, dt$, to \tl{an oriented integral} varifold $\mu$, i.e., a \tl{non-negative} measure on $(\tilde \nu,x,t) \in \mathbb{S}^{d-1}\times \R^d\times [0,+\infty)$,
		\tl{which satisfies the compatibility condition
		\begin{equation}
			-\nabla \chi_{E(t)} 
			= \nu_{E(t)}|\nabla \chi_{E(t)}| 
			=\int_{\mathbb{S}^{d-1}} \tilde \nu \, \mu_t(d\tilde \nu, \,\cdot\, )
			\quad \text{in the sense of measures}
		\end{equation}
		and in particular} $\left|\nabla \chi_{E(t)} \right|\le \tl{\int_{\mathbb{S}^{d-1}}  \mu_t(d\tilde \nu, \,\cdot\,)}$.
		\tl{Here and throughout, $\nu_{E(t)}=\tl{-} \frac{\nabla \chi_{E(t)}}{|\nabla\chi_{E(t)}|}$ denotes the (measure theoretic) \tl{outer} unit normal of $E(t)$.}
	
		\tl{Furthermore, the tuple $(E,\mathbf{u}, \mu)$ satisfies the distributional equation}
%		 if we assume in addition that for some fixed finite time horizon $T<+\infty$
%		\[ 
%			\lim_{h\downarrow 0} \int_0^T P(E_h(t))\,dt = \int_0^T P(E(t))\,dt,
%		\] 
%		then
		\begin{equation}\label{eq:MS_2}
			\begin{split}- \int_0^{\infty} \int_{E(t)} \mathbf{u}\cdot \xi \,dxdt 
			=&  \int_0^{\infty} \int \int \left( \Div \xi - \tilde \nu\cdot D\xi \, \tilde \nu \right) d\mu_t(\tilde \nu,x)\,dt\\
			&+ 2 \int_0^{\infty} \int k \ast \chi_{E(t)}\, \xi \cdot \nu_{E(t)} \left|\nabla \chi_{E(t)}\right| dt
			\end{split}
		\end{equation}
		for all $\xi \in C_0^\infty(\R^d\times (0,+\infty),\R^d)$ with $ \Div \xi =0$, as well as the optimal energy dissipation relation
		\begin{align}\label{eq:MS-DG}
		&\mu_T (\mathbb{S}^{d-1}\tl{\times }\R^d) + \int_{E(T)} k\ast \chi_{E(T)}\,dx
			+\frac12 \int_0^T\int_{E(t)} |\mathbf{u}(x,t)|^2 \,dx \,dt\notag\\
			&+  \int_0^T \int \int \left( \Div \xi -\tilde \nu\cdot D\xi \,\tilde \nu \right) d\mu_t(\tilde \nu, x)\,dt  +2 \int_0^T\int k\ast \chi_{E(t)}\,  \xi \cdot\nu_{\tl{E(t)}} \left|\nabla \chi_{E(t)} \right| dt\notag \\
			&\qquad \qquad \qquad \qquad\qquad \qquad \qquad \qquad- \frac12 \int_0^T\int_{E(t)} \left|\xi\right|^2 dx\,dt
			 \leq \E(E_0)
		\end{align}
		for almost all $T<+\infty$ and all $\xi\in C_0^\infty( \R^d\times (0,+\infty),\R^d)$ with $\Div \xi =0$.
\end{theorem}

		\begin{remark}
			The \tl{system} of equations derived in the theorem \tl{is} indeed a weak form of the free boundary problem \eqref{eq:MS_1a_strong}--\eqref{eq:MS_2b_strong} \tl{provided the sets $E(t)$ are (essentially) open}: 
		\begin{enumerate}[(i)]
	\item		The continuity equation \eqref{eq:MS_1} encodes \eqref{eq:MS_1a_strong} \& \eqref{eq:MS_1b_strong}  as well as the initial conditions $E_0$.
			
	\item	\tl{Under the assumption that no hidden boundary lies inside of $E(t)$, i.e., $\operatorname{spt} \mu_t \cap E(t) = \emptyset$,  Equation \eqref{eq:MS_2} encodes both \eqref{eq:MS_2a_strong} (since \eqref{eq:MS_2} says that $\mathbf{u}$ is orthogonal to divergence-free fields in $E(t)$) and the balance of forces \eqref{eq:MS_2b_strong} on the free boundary.}
		
	\item	The last three left-hand side terms involving the test vector field $\xi$ in  \eqref{eq:MS-DG} can be viewed as a fractional Sobolev norm of $H +2k\ast \chi_E$, and \eqref{eq:MS-DG} is a type of De Giorgi inequality, which for a smooth gradient flow characterizes the solution. 
	\end{enumerate}
	\end{remark}
	\begin{remark}
	Note also that we may replace the sum of the first two left-hand side terms in \eqref{eq:MS-DG} by the (smaller) energy $\E(E(T))$. 
		\tl{Let us assume for a moment that the energies converge as $h\to 0$, which in view of the continuity of the non-local term and the lower semi-continuity of the perimeter is equivalent to saying that the perimeters do not drop down as $h\to0$, i.e., }
		\[ 
				\limsup_{h\downarrow 0} \int_0^T P(E_h(t))\,dt \leq \int_0^T P(E(t))\,dt.
		\] 
		\tl{Then} we may replace the measure $\mu_t$ by the $BV$-version $\delta_{\nu_{E(t)}} \otimes \left|\nabla \chi_{E(t)}\right|$ in all terms appearing in \eqref{eq:MS_2} \& \eqref{eq:MS-DG}. \tl{Indeed, in that case, the convergence of the curvature term
		\begin{align*}
			\lim_{h\downarrow0}\int_0^T\int &\left( \Div \xi - \nu_{E_h(t)}\cdot D\xi \nu_{E_h(t)} \right) |\nabla \chi_{E_h(t)}|\\
			&\qquad\qquad=\int_0^T\int \left( \Div \xi - \nu_{E(t)}\cdot D\xi \nu_{E(t)} \right) |\nabla \chi_{E(t)}|
		\end{align*}
		follows directly from Reshetnyak's continuity theorem, see e.g.\ \cite[Theorem 20.12]{maggi2012sets}.}
	\end{remark}
	
\tl{ \begin{remark}
		The more precise structure of the varifold $\mu_t$ is not clear. The integrality (or even rectifiability) of $\mu_t$ away from $\operatorname{supp} |\nabla \chi_{E(t)}|$ does not simply follow from curvature bounds and the control in time on the sets $E_h(t)$. Indeed, it is easy to construct counterexamples for which $\mu_t$ oscillates in time but both estimates are valid.
	\end{remark}

	\begin{remark}
		The optimal energy-dissipation rate (here in form of \eqref{eq:MS-halfDG} or \eqref{eq:MS-DG}) plays a crucial role in recent weak-strong uniqueness proofs and does not follow from the weak formulation \eqref{eq:MS_1} \tl{\& \eqref{eq:MS_2}}.
    \end{remark}

}
	In the following we write $A\lesssim B$ if there exists a generic constant $C=C(d)$ such that $A\leq C\,B$. 
	\section{Compactness}\label{sec:compactness}
\begin{lemma}[Compactness]
		Suppose $E_0$ satisfies \eqref{eq:standingassumption} and let $E_h$ be constructed by the scheme as above. Then
		\begin{equation}\label{eq:W2cont}
			W_2(\chi_{E_h(t)},\chi_{E_h(s)}) \lesssim \E(E_0)^{\frac12} (t-s)^\frac12
		\end{equation}
		and
		\begin{equation}\label{eq:L1cont}
			\left| E_h(t) \triangle E_h(s)\right| \lesssim \E(E_0)^{\frac34} (t-s)^{\frac14}
		\end{equation}
		for all $t>s\geq0$ with $t-s \geq h$. 
		
		Therefore, there exists a subsequence $h\downarrow0$ and a one-parameter family of finite perimeter sets $(E(t))_{t\geq0}$ such that for any $T<+\infty$
				\begin{equation}\label{eq:supconv}
					\sup_{t\in [0,T]} |E_h(t)\triangle E(t)|  \to 0\quad \text{ as }h\downarrow0.
				\end{equation}
				
		Furthermore, the limit satisfies 
		\begin{equation}\label{eq:W2contlimit}
			W_2(\chi_{E(t)},\chi_{E(s)}) \lesssim \E(E_0)^{\frac12} (t-s)^\frac12
		\end{equation}
		and
		\begin{equation}\label{eq:L1contlimit}
			\left| E(t) \triangle E(s)\right| \lesssim  \E(E_0)^{\frac34} (t-s)^{\frac14}
		\end{equation}
		for all $t>s\geq0$.
\end{lemma}

\begin{proof}
	Using $E_{n-1}$ as a competitor in \eqref{eq:mm} yields
	\[
	\frac{1}{2h}W_2^2(\chi_{E_n},\chi_{E_{n-1}}) + \E(E_n)
	\le  \E(E_{n-1}),
	\]
	so that after summation in $n$ and telescoping
	\begin{equation}\label{eq:energybound}
	\frac{h}{2} \sum_{n=n_0+1}^{n_1} \left(\frac{W_2(\chi_{E_n},\chi_{E_{n-1}})}{h}\right)^2 + \E(E_{n_1})\le \E(E_{n_0}).
	\end{equation}
	In particular, for any pair of integers $n_1>n_0 \geq0$, we have
	\begin{align*}
	W_2(\chi_{E_{n_0}},\chi_{E_{n_1}}) &\le \sqrt{ (n_1-n_0)h} \left(\sum_{n=n_0+1}^{n_1}
	\frac{1}{h}W_2(\chi_{E_n},\chi_{E_{n-1}})^2\right)^{\frac12}
	\\ & \le \sqrt{(n_1-n_0)h} \sqrt{2(\E(E_{n_0})-\E(E_{n_1}))},
	\end{align*}
	which implies \eqref{eq:W2cont}.
	
	The $L^1$ estimate \eqref{eq:L1cont} then follows from \eqref{eq:W2cont} in conjunction with the interpolation inequality in Corollary \ref{cor:interpolation2} \tl{in Section \ref{sec:interpolation inequality} below} and Jensen's inequality in the form of $W_1(\chi,\tilde\chi) \leq W_2(\chi,\tilde\chi)$.
	
	The energy estimate \eqref{eq:energybound} also yields a uniform bound on the perimeter, \tl{hence}
	\[
	\int |\chi_{E_h(t)}(x+z) - \chi_{E_h(t)}(x) | \,dx \leq |z| P(E_h(t)) \leq |z| \E(E_0),
	\]
	i.e., we have a uniform modulus of continuity in space. Together with the uniform modulus of continuity in time \eqref{eq:L1cont}, which is valid down to scales $h$, this allows us to apply the Riesz-Kolmogorov compactness theorem in $L^1([0,T]\times K)$ for any compact set $K\subset \R^d$ and any $T<+\infty$. A diagonal argument yields
	$\chi_{E_h} \to \chi_E $ in $L^1_{\textup{loc}}([0,\infty)\times \R^d)$. But since $\int |x|^2\chi_{E_h(t)}(x)dx<\infty$, which follows from \eqref{eq:standingassumption} \& \eqref{eq:W2cont}, this implies the $L^1$-convergence globally in space, and locally in time.
        Eventually, an Ascoli-Arzel\`a type argument, together with the estimate
        in Corollary~\ref{cor:interpolation2},
        allows to deduce the local uniform convergence in time, i.e., \eqref{eq:supconv}.
	 The continuity estimates \eqref{eq:W2contlimit} \& \eqref{eq:L1contlimit} then follow immediately.	
\end{proof}

\section{Convergence}\label{sec:convergence}
\begin{proof}[Proof of Theorem \ref{thm}]\hfill 
	
\nopagebreak

%%%%%%%%%%%%%%%%%%%%%%%%%%%%%%%%%%
%%%%%%%%%%%%%%%%%%%%%%%%%%%%%%%%%%
\noindent\emph{Step 1: Construction of $\mathbf{u}$ and verification of \eqref{eq:MS_1}.}
By Kantorovich duality
\[
\frac{1}{2h}W^2_2(\chi_{E_n},\chi_{E_{n-1}})
= \sup_{\phi(x)+\psi(y)\le \frac{|x-y|^2}{2h}} \int_{E_n}\phi(x)\, dx + \int_{E_{n-1}}\psi(y)\, dy.
\]
This supremum is reached at $(\phi^n,\psi^n)$ such that
\[
\Phi^n (x) = \frac{|x|^2}{2}-h\phi^n(x),\quad
\Psi^n(y) = \frac{|y|^2}{2} -h\psi^n(y)
\]
are convex conjugates and $(\nabla\Phi^n)_\sharp \chi_{E_n}=\chi_{E_{n-1}}$ \tl{solves the optimal transportation problem \eqref{eq:def W2} defining the distance}  $W_2(\chi_{E_n},\chi_{E_{n-1}})$, see \cite[Theorem 1.3]{brenier1991polar} or \cite[Theorem 2.12]{villani2003topics}. 
%In particular $\nabla\Psi^n$ is bounded (near $E_n$,
%or in $\R^d$ I think if we really want).  
 Hence
\begin{equation}\label{eq:W2asDpsi}
\frac{1}{2h}W_2^2(\chi_{E_n},\chi_{E_{n-1}}) = h\int_{E_n}|\nabla\phi^n(x)|^2 dx
= h\int_{E_{n-1}}|\nabla\psi^n(y)|^2 dy
\end{equation}
%and $\det (I-h\nabla^2 \phi^n) = 1$ a.e.~in $E_n$. 
 and by ~\eqref{eq:energybound}
\begin{equation}\label{eq:energybound2}
h \sum_{n=n_0+1}^{n_1} \int_{E_n}|\nabla\phi^n(x)|^2 dx\le \E(E_{n_0})-\E(E_{n_1}) \le  \E(E_0),
\end{equation}
that is, if we set $\mathbf{u}^n:=\chi_{E_n}\nabla\phi^n$ and $\mathbf{u}_h(t) := \mathbf{u}^{[t/h]}$, then  $\mathbf{u}_h$ is uniformly  bounded in $L^2$. Let
$\mathbf{u}=\mathbf{u}(x,t)$ be a weak limit. Since $(Id-h\nabla \phi^n)_\sharp \chi_{E_n}=\chi_{E_{n-1}}$,  for $\eta(x,t)$ a smooth
test function
\begin{align*}
\frac{1}{h}\int_{\R^d}(\chi_{E_n}-\chi_{E_{n-1}})\,\eta\, dx
= \frac{1}{h} \int_{\R^d} \chi_{E_n}(x)\, (\eta(x)-\eta(x-h\nabla\phi^n(x))) \, dx.
\end{align*}
Using Taylor's theorem in the form $\big|\eta(x)-\eta(x-h\xi) -h\xi \cdot \nabla \eta(x) \big| \leq \frac{h^2}{2} |\xi|^2 \sup_x | \nabla^2 \eta|$, we can replace the right-hand side by
\[
\int_{\R^d} 
\nabla\eta(x)\cdot\nabla\phi^n(x)\chi_{E_n} (x)dx
\]
at the expense of the error
\[
\frac1h \frac{h^2}2 \sup | \nabla^2\eta|  \int_{\R^d} |\nabla\phi^n|^2 \chi_{E_n} \,dx = \sup | \nabla^2\eta|   \frac1{2h} W_2^2(\chi_{E_n},\chi_{E_{n-1}}).
\]
After integration in time, this error term vanishes as $h\downarrow0$ because of \eqref{eq:energybound} and we may pass to the limit in the time-integrated version of the above identity to obtain the continuity equation in form of \eqref{eq:MS_1}. 
%So the question is to identify $\mathbf{u}$: basically, show that $\mathbf{u}=\nabla\phi$
%with $\phi=\kappa_E$ (up to a constant) on $\partial E$.

%%%%%%%%%%%%%%%%%%%%%%%%%%%%%%%%%%
%%%%%%%%%%%%%%%%%%%%%%%%%%%%%%%%%%

\smallskip

%%%%%%%%%%%%%%%%%%%%%%%%%%%%%%%%%%
%%%%%%%%%%%%%%%%%%%%%%%%%%%%%%%%%%

\noindent
\emph{Step 2: De Giorgi's interpolation and argument for \eqref{eq:MS-halfDG}.} 
De Giorgi's variational interpolation
\begin{equation}\label{eq: def de giorgi interpolation}
\tilde E_h((n-1)h+t) \in \arg \min_E \Big\{ \frac1{2t} W_2^2(\chi_E,\chi_{E_{n-1}}) + \E(E) \Big\}
\end{equation}
satisfies the identity
\begin{equation}\label{eq:de giorgi one step}
\begin{split}
\frac h2 \bigg( \frac{W_2(\chi_{E_n},\chi_{E_{n-1}})}{h}\bigg)^2 +\frac12 \int_{(n-1)h}^{n h}& \bigg( \frac{W_2(\chi_{\tilde E_h(t+(n-1)h)},\chi_{E_{n-1}})}{t}\bigg)^2 dt\\
  &\qquad \qquad \qquad \leq  \E(E_{n-1}) - \E(E_{n}).
  \end{split}
\end{equation}
Although the proof is contained---in a more general context---in \cite[Theorem 3.1.4]{AGS}, we repeat it here for the reader's convenience.

W.l.o.g.\ we may assume $n=1$; for notational convenience we also drop the index $h$ for this short argument. Defining momentarily
\[
f(t):= \frac1{2t} W_2^2(\chi_{\tilde E(t)}, \chi_{E_0}) + \E(\tilde E(t))
\]
to be the minimal value in the variational problem \eqref{eq: def de giorgi interpolation}, we may compute for $s<t$, using the minimality of $\tilde E(t)$,
\begin{align*}
	f(t)-f(s) 
	&\leq \frac1{2t} W_2^2(\chi_{\tilde E(s)}, \chi_{E_0}) +\E(E(s))- \frac1{2s} W_2^2(\chi_{\tilde E(s)}, \chi_{E_0}) -\E(E(s))\\
	&= \frac{s-t}{2st} W_2^2(\chi_{\tilde E(s)}, \chi_{E_0}).
\end{align*}
Since $s<t$, this implies
\[	
	\frac{f(t)-f(s)}{t-s} \leq - \frac1{2st} W_2^2(\chi_{\tilde E(s)}, \chi_{E_0}) \to - \frac1{2t^2} W_2^2(\chi_{\tilde E(t)}, \chi_{E_0})  \quad \text{as } s\uparrow t.
\]
The analogous reverse inequality may be obtained by using $s>t$ in the above argument with the roles of $s$ and $t$ interchanged. 
Hence $f$ is locally Lipschitz in $(0,h]$ with 
\[
\frac{d}{dt} f(t)  = - \frac1{2t^2} W_2^2(\chi_{\tilde E(t)}, \chi_{E_0})
\]
for almost every $t\in(0,h)$.
For $\varepsilon>0$, integrating this inequality from $t=\varepsilon $ to $t=h$, and then using lower semicontinuity w.r.t.\ the $L^1$ convergence $E(\varepsilon) \to E_0$ yields \eqref{eq:de giorgi one step}.
	
Summing \eqref{eq:de giorgi one step} over $n$ from $n_0+1$ to $n_1$ and telescoping the right-hand side we obtain the sharp energy dissipation inequality:
\begin{equation}\label{eq:DGforDiscr}
\begin{split}
\frac h2  \sum_{n=n_0+1}^{n_1} \bigg( \frac{W_2(\chi_{E_n},\chi_{E_{n-1}})}{h}\bigg)^2 +\frac12 \int_{n_0 h}^{n_1 h} &\bigg( \frac{W_2(\chi_{\tilde E_h(t)},\chi_{E_{h}(t)})}{t-h[t/h]}\bigg)^2 dt\\
&\qquad \quad \leq \E(E_{n_0}) - \E(E_{n_1}).
\end{split}
\end{equation}

By \eqref{eq:W2asDpsi} we have
\[
 \frac{h}{2}\sum_{n=1}^{N} \left(\frac{W_2(\chi_{E_n},\chi_{E_{n-1}})}{h}\right)^2 =  \frac{h}{2}\sum_{n=1}^{N} \int \left|\chi_{E_n}\nabla \phi^n\right|^2dx,
\]
which implies
\begin{equation}\label{eq:lsc_metric}
\int_0^T \int_{E(t)} |\mathbf{u}(x,t)|^2 \,dx\,dt \leq
\liminf_{h\to0}\, \frac{h}{2}\sum_{n=1}^{N} \left(\frac{W_2(\chi_{E_n},\chi_{E_{n-1}})}{h}\right)^2 
\end{equation}
since $\chi_{E_n}\nabla \phi^n = \mathbf{u}^n$ and $\mathbf{u}_h \rightharpoonup \mathbf{u}$ in $L^2$.

Following the same strategy as in Step 1, we can show that $\tilde{\mathbf{u}}_h:=\chi_{\tilde E_h(t)} \nabla \tilde \phi_h(t)$ with $x-(t-h[t/h])\nabla \tilde \phi_h(x,t)$ optimal in $W_2(\chi_{\tilde E_h(t)},\chi_{E_h(t)})$---after passage to a subsequence---weakly converges to the same limit $\mathbf{u} = w-\lim_{h\downarrow0} \mathbf{u}_h$. 
In particular, as before,
\[
\frac{1}{t-h[t/h]}W_2^2(\chi_{\tilde E_h(t)},\chi_{E_{h}(t)}) 
=(t-h[t/h])\int |\nabla \tilde \phi|^2 dx
\]
so that after division by $(t-h[t/h])$ and integration in $t$
\[
\frac12 \int_{0}^{T} \bigg( \frac{W_2(\chi_{\tilde E_h(t)},\chi_{E_{h}(t)})}{t-h[t/h]}\bigg)^2 dt
= \frac12\int_0^T \int |\nabla \tilde \phi_h|^2dx\,dt
 = \frac12 \int_0^T \int | \tilde{\mathbf{u}}_h|^2dx\,dt,
\]
which is again lower semi-continuous. 
This concludes the argument for \eqref{eq:MS-halfDG} \tl{for a.e.\ $T>0$.
	Now let $T>0$ be arbitrary and let $T_n\to T$ such that \eqref{eq:MS-halfDG} holds for each $T_n$. Since $E(T_n)\to E(T)$ in $L^1$, the first term on the left-hand side of \eqref{eq:MS-halfDG} is lower-semicontinuous. The second one is clearly continuous in $T$ so the inequality holds for \emph{all} $T>0$.} 
\tl{We will prove the refined energy dissipation inequality \eqref{eq:MS-DG} later on in Step~4.}

% If instead of passing to the limit $h\downarrow0$ in the inequality \eqref{eq:delE} as in Step 3, we first use the Euler-Lagrange equation \eqref{eq:ELeq} for the right-hand side to obtain
% \[
% |\partial \E|(E(t)) \geq     \int_{E_n} \nabla \phi^n \cdot \xi \,dx
%- \frac12  \int_{E_n} |\xi|^2 dx,
% \]
% we arrive at the desired \eqref{eq:MS-halfDG} by first passing to the limit $h \downarrow 0$ and then taking the supremum in $\xi$.
% Note that $\E(E(t)) \leq \liminf_{h\downarrow0} \E(E_h(t))$ due to the lower semi-continuity of the perimeter and the continuity of the non-local term, see Step 2.

%%%%%%%%%%%%%%%%%%%%%%%%%%%%%%%%%%
%%%%%%%%%%%%%%%%%%%%%%%%%%%%%%%%%%
\smallskip
%%%%%%%%%%%%%%%%%%%%%%%%%%%%%%%%%%
%%%%%%%%%%%%%%%%%%%%%%%%%%%%%%%%%%

\noindent
\emph{Step 3: Derivation of \eqref{eq:MS_2}.} 
%
%
%\subsection{Using the perimeter}
%One has for $F$ with $|F|=|E|$ and $(\phi_F,\psi_F)$ the MK potentials
%of $W_2(\chi_E,\chi_{E_{n-1}})$:
%\[
%\E(E_n) + \int_{E_n}\phi_F dx \le \E(F) + \int_{F}\phi_F dx,
%\]
%or
%\[
%\E(E_n) + \int_{E_n}\phi^n dx \le \E(F) + \int_{F}\phi^n dx
%+ \int (\chi_F-\chi_{E_n}) (\phi_F-\phi^n)dx.
%\]
%In particular if $F\triangle E_n\subset B_\rho(x)$,
%\begin{equation}\label{eq:minimality}
%\E(E_n) +\int_{E_n}\phi^n dx 
%\le \E(F) +  \int_{F}\phi^n dx + o(\rho^n)
%\end{equation}
%so that $E_n$ should be a smooth set (ref?) with curvature
% $\kappa_{E_n} = -\phi^n(x)+ \textup{const.}$. (In dimension $\le 7$;
%for $k\neq 0$ one also should have a term in $k*\chi_E$...)
%This function is Lipschitz so that $\partial E_n$ is $C^{1,\alpha}$ 
%($C^{1,1}$ in dimension 2). A problem is of course to find an estimate which
%is true as $h\to 0$.
%
% For instance if one knew that $\int|\nabla \phi^n|^2dx$  is bounded in
% a neighborhood of the sets (showing first they are bounded)
% then in dimension $\le 3$ (or $4$) one might
% find a uniform density estimate for the sets $E_h(t)$ which
% would converge Hausdorff.
The Euler-Lagrange equation of the minimization problem \eqref{eq:mm} reads
\begin{equation}\label{eq:ELeq}
 - \int_{E_n}  \nabla \phi^n \cdot \xi \,dx = \int \left( \Div \xi - \nu_{E_n} \cdot D\xi \nu_{E_n} + 2k\ast \chi_{E_n}
\xi \cdot \nu_{E_n} \right) \left| \nabla \chi_{E_n}\right| 
\end{equation}
 for all smooth test vector fields $\xi$ with $\Div \xi=0$, where
$\nu_{E_n} = -\frac{\nabla \chi_{E_n}}{|\nabla \chi_{E_n}|}$ denotes the \tl{outer} normal. 

%After integration in time, the left-hand side converges to $\int_0^T \int_E \mathbf{u} \cdot \xi \,dxdt$ thanks to Step 1. 
% 
%After integration in time, $\int_0^T\int \Div\xi|\nabla \chi_{E_h}|\,dt\to \int_0^T \int \Div\xi|\nabla \chi_{E}|\,dt$
%and 
%\[
%\int_0^T \int\nu_{E_h}\cdot D\xi\, \nu_{E_h} |\nabla \chi_{E_h}| \,dt\to
%\int_0^T\int\nu_E\cdot D\xi\,\nu_E |\nabla \chi_{E}|\,dt
%\]
%by Reshetnyak's continuity theorem (\textit{cf.}\ ~\cite[Theorem 2.39]{AFP}).

The non-local term
\[
\int_0^\infty \int 2k\ast \chi_{E_h(t)} \xi \cdot \nu_{E_h(t)} \left| \nabla \chi_{E_h(t)}\right|dt = \tl{-}
\int_0^\infty \int 2k\ast \chi_{E_h} \xi \cdot \nabla \chi_{E_h}
\]
converges since $\nabla \chi_{E_h} \to \nabla \chi_E$ weakly as measures, and since $k \ast \chi_{E_h(t)}$ converges uniformly, which follows from the strong $L^1$ convergence $\chi_{E_h} \to \chi_E$ and the observation that
\[
\sup |k\ast \chi - k \ast \tilde \chi | \le \int | \chi - \tilde \chi| \,dx
\]
 for any two characteristic functions $\chi,\tilde \chi$ for which only the integrability $\int k =1$ and non-negativity $k\geq0$ are needed.
 
 Since the measures $\mu^h= \mu^h_t dt=\delta_{\nu_{E_h(t)}(x)}\otimes \left|\nabla \chi_{E_h(t)}\right|dt$ are bounded
 (with moreover $\mu^h(\mathbb{S}^{d-1}\times \R^d\times I)\le P(E_0)|I|$ for any $I\subset (0,+\infty)$ measurable), by Banach-Alaoglu, they have a weak-$\ast$ limit $\mu = \mu_tdt$ (after passage to a subsequence).
  Hence we can identify the limit of the first right-hand side term in \eqref{eq:ELeq} as well:
 \[
	 \begin{split}
	 \lim_{h\downarrow0} \int_0^\infty \int (\Div \xi &-\nu_{E_h}\cdot D\xi\, \nu_{E_h}) |\nabla \chi_{E_h}| \,dt\\
	 =&\lim_{h\downarrow0} \int_0^\infty \int \int (\Div \xi -\tilde \nu\cdot D\xi\, \tilde \nu)\,d\mu^h_t(\tilde\nu,x)dt\\
	 = & \int_0^\infty \int \int (\Div \xi -\tilde \nu \cdot D\xi\, \tilde \nu) \,d\mu_t(\tilde\nu,x) dt
	 \end{split}
 \]
 for any test vector field $\xi\in C_0^\infty(\R^d\times(0,+\infty),\R^d)$.
%%%%%%%%%%%%%%%%%%%%%%%%%%%%%%%%%%
%%%%%%%%%%%%%%%%%%%%%%%%%%%%%%%%%%
\smallskip
%%%%%%%%%%%%%%%%%%%%%%%%%%%%%%%%%%
%%%%%%%%%%%%%%%%%%%%%%%%%%%%%%%%%%

\noindent
\emph{Step 4: Proof of the optimal energy dissipation relation \eqref{eq:MS-DG}.} 
The local slope of $\E$, defined via
\[
|\partial \E(E)| := \limsup_{F\to E} \frac{(\E(E)-\E(F))_+}{W_2(\chi_{E},\chi_{F})},
\]
where the convergence of the sets $F\to E$ is to be understood with respect
to $W_2$, % i.e., $W_2(\chi_E,\chi_F) \to 0$,
satisfies
\[
|\partial \E| (\tilde E((n-1)h+t)) \le \frac{W_2(\chi_{\tilde E((n-1)h+t)},\chi_{E_{n-1}})}{t},
\]
cf.\ \cite[Lemma 3.1.3]{AGS}. 
Applying this to \eqref{eq:DGforDiscr} yields the sharp energy dissipation inequality
\begin{equation}\label{eq:sharpenergyinequality}
\begin{split}
\frac h2 \sum_{n=n_0+1}^{n_1} \left( \frac{W_2(\chi_{E_n},\chi_{E_{n-1}})}{h}\right)^2 + \frac12 \int_{n_0h}^{n_1h} &\left| \partial \E\right|^2(\tilde E_h(t)) \,dt \\
&\qquad \quad \le \E(E_h (n_0h)) - \E(E_h(n_1h)).
\end{split}
\end{equation}

%%%%%%%%%%%%%%%%%%%%%%%%%%%%%%%%%%
%%%%%%%%%%%%%%%%%%%%%%%%%%%%%%%%%%

Our goal is to pass to the limit in \eqref{eq:sharpenergyinequality}.
We have already done this for the metric term in Step 2. 
\tl{Rewriting the surface energy as $P(E_h(t)) = \mu^h_t(\mathbb{S}^{d-1}\times \R^d)$, it is clear that the right-hand side terms converge to the desired limits. 
	It remains to show}
\begin{equation}\label{eq:lsc_dE}
\begin{split}
	\int_0^T  \int \left( \Div \xi -\tilde \nu\cdot D\xi \,\tilde \nu\right) \,& d\mu_t(\tilde\nu,x)dt + \int_0^T\int 2 k\ast \chi_{E(t)}\, \xi \cdot \nu_{E(t)}  \,|\nabla \chi_{E(t)} |\,dt \\
	- \frac12 & \int_0^T \int_{E(t)} \left|\xi\right|^2 dx\, dt \leq \liminf_{h\to0}  \frac{1}{2} \int_{0}^{T}|\partial \E|^2 (\tilde E_{\tl{h}}(t))\,  dt
	\end{split}
\end{equation}
for all test vector fields $\xi $ with $\Div \xi =0$.

%%%%%%%%%%%%%%%%%%%%%%%%%%%%%%%%%%
%%%%%%%%%%%%%%%%%%%%%%%%%%%%%%%%%%

Given \tl{an arbitrary set of finite perimeter $E$  satisfying \eqref{eq:standingassumption}}, any smooth divergence free vector field $\xi$ provides a one-parameter family of candidates for the $\limsup$ in the definition of the local slope $|\partial \E|(E)$ at $E$ via the inner variations $\partial_s \chi_{E_s} +\xi \cdot \nabla \chi_{E_s} =0$.
\tl{Using the elementary relation $0 \leq \frac12 (\frac ab-b)^2= \frac12 (\frac ab)^2 -a + \frac12 b^2$, this yields}
\[
\begin{split}\frac12|\partial \E|^2(E) 
&\geq \lim_{s\to0} \frac12\left(\frac{\frac1s\left(\E(E)-\E(E_s)\right)_+}{\frac1sW_2(\chi_E,\chi_{E_s})}\right)^2\\
&\geq  \lim_{s\to0}\frac1s \left(\E(E)-\E(E_s)\right)_+ -  \frac1{2s^2}W_2^2(\chi_E,\chi_{E_s}).
\end{split}
\]
On the one hand, since $\xi$ is divergence free, it generates one particular volume-preserving flow from $E$ to $E_s$. More precisely, the rescaled field $s\xi$ solves $\partial_{s'} \chi_{E_{s'}} + \Div( s\xi \chi_{E_{s'}}) =0 $ and transports $E$ to $E_s$ in one unit of time and hence provides a particular candidate for the minimum problem in $W_2$:
\[
W^2_2(\chi_E,\chi_{E_s}) \leq  \int_0^1 \int_{E_{s'}} |s\xi|^2 \,dx\,ds'  = s^2 \int_{E} |\xi|^2 dx  +o(s^2).
\]
On the other hand, we have
\begin{align*}
(\E(E) -\E(E_s))_+
&\geq \E(E) -\E(E_s) \\
&= -s \frac{d}{ds}\Big|_{s=0} \E(E_s) +o(s)\\
&= -s \int\left( \Div \xi - \nu_E \cdot D\xi \, \nu_E \right) \tl{|\nabla \chi_E|} \\
&\quad -2s  \int  k\ast \chi_{E}\, \xi \cdot \nu_{E}  \left| \nabla \chi_E\right|  +o(s).
\end{align*}
Therefore, \tl{for any $E$ satisfying \eqref{eq:standingassumption}, and any test vector field field $\xi$ with $\Div \xi=0$ it holds} 
\begin{equation}\label{eq:delE}
	\begin{split}
		\frac12 |\partial \E|^2(E) \geq  &-\int\left( \Div \xi - \nu_E \cdot D\xi \, \nu_E \right) |\nabla \chi_E|\\&
		 -2 \int  k\ast \chi_{E} \xi \cdot \nu_{E}  \left| \nabla \chi_E\right| 
		- \frac12  \int_{E} |\xi|^2 dx.
	\end{split}
\end{equation}
\tl{Replacing $\xi$ by $-\xi$ and applying this argument to $E_h(t)$, writing the first integral in terms of $\mu^h_t$ as 
	\[
	\int\left( \Div \xi - \nu_{E_h(t)} \cdot D\xi \, \nu_{E_h(t)} \right) |\nabla \chi_{E_h(t)}|
	= \int \left(\Div \xi - \tilde \nu \cdot D\xi \tilde \nu\right) d\mu^h_t(x,\tilde \nu),
	\]
	integrating in $t$, and} taking the limit $h\to0$ yields \eqref{eq:lsc_dE}.
\end{proof}

%\subsection{What about $k$}
%
%In~\cite{Otto-Labyrinthine},
%\[
%k = \frac{1}{2\pi}\left(\frac{1}{|x|}-\frac{1}{\sqrt{1+|x|^2}}\right)
%\]
%in dimension $d=2$.
%% \[ 
%% \int_{\R^2} k(x)dx
%% = \lim_{r\to\infty} \int_0^r 1-\frac{r}{\sqrt{1+r^2}} dr
%% = \lim_{r\to\infty} (r - (\sqrt{1+r^2}-1)) =  1.
%% \]
%
%
%I guess one should write~\eqref{eq:minimality} as
%\begin{multline*}
%\Per(E_n)+\int_{E_n} f dx + 2\int_{E_n} k*\chi_{E_n} dx
%\\\le 
%\Per(F)+\int_{F} f dx + 2\int_{F} k*\chi_{E_n}dx  
%+ \int (\chi_F-\chi_{E_n})k*(\chi_F-\chi_{E_n}) dx+
%o(\rho^n).
%\end{multline*}
%Then, estimate, for $F\triangle E_n\subset B_\rho(x)$,
%assuming $|k|\le C/|x|^{d-s}$, $s>0$ (this excludes fractional
%perimeters)
%\begin{multline*}
% \int (\chi_F-\chi_{E_n})k*(\chi_F-\chi_{E_n}) 
%\le\int_{B_\rho} \int_{B_\rho} k(x-y)dydx
%\\\le \int_{B_\rho}\int_{B_{2\rho}}k dydx
%\le C\rho^{d+s} ds = o(\rho).
%\end{multline*}
%
%Hence we end up with the same regularity, and the curvature
%should be $\kappa_{E_n}(x) = -\phi^n(x) - 2k*\chi_{E_n}(x) + \textup{const.}$.
%One should obtain
%\[
%-\int_{\R^d} \chi_{E_n} \nabla\phi^n \cdot\xi dx
% = 
%\int \left(\Div \xi -\nu_{E_n}\cdot (D\xi \nu_{E_n})\right)|D\chi_{E_n}|
%+ 2\int k*\chi_{E_n}\xi\cdot D\chi_{E_n}
%\]
%and in the limit,
%\[
%-\int_{\R^d} \chi_{E} \nabla\phi \cdot\xi dx
% = 
%\int \left(\Div \xi -\nu_{E}\cdot (D\xi \nu_{E})\right)|D\chi_{E}|
%+ 2\int k*\chi_{E}\xi\cdot D\chi_{E}
%\]
%which is~\eqref{eq:evolutionv}.

\section{An interpolation inequality}\label{sec:interpolation inequality}

\begin{lemma}\label{lem:interpolation2}
  There exists $C>0$ such that for $u,v\in BV(\R^d)$ with $\int (|u|+|v|)|x|dx<+\infty$,
  one has
  \begin{equation*}
    \|u-v\|_{L^1}^2 \le C  \|u-v\|_{W^{-1,1}} |D(u-v)|(\R^d).
  \end{equation*}
\end{lemma}

\begin{proof}
  We first consider, ${f},{g}$ smooth with compact support, a symmetric
  mollifier $\rho$, and $\sigma>0$.  We have
  \[
    \int {f} {g} \,dx = \int {g}\,(\rho_\sigma*{f}) \,dx +
    \int {g} \,({f}-\rho_\sigma*{f})\,dx .
  \]
  For the first integral we use
  $ %\[
    |\nabla (\rho_\sigma*{g})| \le (\|\nabla\rho\|_{L^1}/\sigma )\|{g}\|_{L^\infty},
  $ %\]
  hence by symmetry of $\rho$:
  \[
    \int {g}\,(\rho_\sigma*{f}) \,dx = \int {f}\, (\rho_\sigma*{g})\,dx \le \frac{C_1}{\sigma}\|{f}\|_{W^{-1,1}} \|{g}\|_{L^\infty},
  \]
  with $C_1=\|\nabla\rho\|_{L^1}$. For the second integral, we write that for ${f}\in BV(\R^d)\cap C^1(\R^d)$,
  \[
    \int |{f}-\rho_\sigma*{f}|dx =\int \left|\int \int_0^\sigma \rho(\xi) \xi\cdot\nabla {f}(x-t\xi)\,dt\,d\xi \right|dx
    \le C_2\sigma \|\nabla {f}\|_{L^1},
  \]
  where $C_2=\int |\xi|\rho(\xi)d\xi$. We deduce that
  \[
    \int {g}({f}-\rho_\sigma*{f})\,dx \le C_2\sigma \|{g}\|_\infty |D{f}|(\R^d).
  \]
  Hence for any $\sigma>0$
  \[
    \int {f} \,{g}\,dx \le  \|{g}\|_{L^\infty}\left(\frac{C_1}{\sigma} \|{f}\|_{W^{-1,1}} 
    + C_2\sigma |D(u-v)|(\R^d)\right)
  \]
  so that (minimizing the right-hand side \tl{with respect to} $\sigma>0$)
  \[
    \int {f}\,{g}\, dx \le  2\|{g}\|_{L^\infty}\sqrt{C_1C_2 \|{f}\|_{W^{-1,1}}  |D{f}|(\R^d)}.
  \]
  Choosing ${g}$ a mollification of $\mathop{\mathrm{sign}} {f}$ and passing to the limit it follows that
  \[
    \int |{f}|\,dx \le  2\sqrt{C_1C_2 \|{f}\|_{W^{-1,1}}  |D{f}|(\R^d)}.
  \]

  This extends to ${f}\in BV(\R^d)$ such that 
  $\|{f}\|_{-1,1}=\sup_{|\nabla{g}|<1}\int{f}{g} \, dx<+\infty$
   and
  $\int |x||{f}(x)|dx<+\infty$, by approximation.
\end{proof}

\tl{Since $\|\chi_E-\chi_F\|_{W^{-1,1}} = W_1(\chi_E,\chi_F)$ we have the following immediate consequence of Lemma \ref{lem:interpolation2}.}
		
\begin{corollary}\label{cor:interpolation2} For any
  sets $E,F\subset\R^d$ with finite perimeter and with $\int_{E\cup F}|x|dx<+\infty$
  \[
    	|E\triangle F| \lesssim \sqrt{\Per(E)+\Per(F)}\sqrt{W_1(\chi_E,\chi_F)}.
  \]
\end{corollary}

\frenchspacing
\bibliographystyle{plain}
\bibliography{lit}

\end{document}